\documentclass[11pt]{article}
\usepackage{times}
\usepackage{amsmath}
\usepackage{amssymb}
\usepackage{amsthm}
\usepackage{graphicx}
\usepackage[small]{caption}
\usepackage{cite}
\usepackage{tikz-cd}

\usepackage{psfrag}
\usepackage{color}

\def\A{\mathcal A} 
\def\R{\mathbb R}

\def\S{\mathcal S}

\def\Diff{{\rm Diff}}

\def\Vect{{\rm Vect}}

\def\D{\mathcal{D}} 
\def\dTO{\mathfrak{D}} 
\def\F{\mathfrak{F}}

\newcommand{\Tr}{\mathrm{Tr}}
\newcommand{\Id}{\mathrm{Id}}
\renewcommand{\div}{\mathop{\mathrm{div}}}

\newtheorem{Def}{Definition}
\newtheorem{theorem}{Theorem}

\author{Dmitry Pavlov \thanks{d.pavlov@imperial.ac.uk} \\ Imperial College London}

\date{\vspace{-5ex}}

\title{Geometric Discretization of the EPDiff Equations}

\begin{document}

\maketitle

\begin{abstract}
In this paper we develop a geometric discretization of the EPDiff equations in
one-dimensional case. We extend the method presented in~\cite{Pavlov:2010dd} to
apply to all (not only divergence-free) vector fields and use a pseudospectral
representation of a vector field. This method can be extended to a
multidimensional case in a straightforward way.
\end{abstract}

\section{Introduction}
The main objective of this paper is to develop a general method of geometric
discretization for infinite-dimensional systems and apply this method to the
EPDiff equation. Geometric integration has been a very large and active area of
research (see~\cite{Marsden:2001dj} for an overview). Unlike conventional
numerical schemes, geometric integrators are derived from variational principles
and preserve the structure of the original systems. The structure-preserving
nature of these methods allows to capture dynamics without usual numerical
artifacts such as energy or momenta loss.

To construct a variational integrator for an infinite-dimensional system, such as
the EPDiff or Euler equations, one first has to develop a method of discretizing
the configuration space of this system, i.e. the group of
diffeomorphisms. Moreover, we have to replace this group with a
finite-dimensional Lie group in order to preserve the symmetries of the original
system. As the second step we can derive a finite-dimensional system on this group
from  Lagrange-D'Alembert principle. Lastly, we apply standard techniques of
variational integration to discretize time and get an update rule.

The method described below extends one developed in~\cite{Pavlov:2010dd} for
incompressible Euler fluids. Here this method is presented in a general case
applicable to all, not only divergence-free, vector fields. Also, a different
(pseudospectral) representation of the velocity field is used. We will apply
this method to the one-dimensional EPDiff equation and present numerical results
in Section~\ref{sec:results}.

\subsection{The EPDiff equations}
The EPDiff equations comprise a family of geodesic equations on the group of
diffeomorphisms $\Diff(M)$ of a manifold $M$, $\dim M = n$, where the metric is
defined by a norm on the space of vector fields $\Vect(M)$ of the following form:
\begin{equation}
  \label{eq:Lnorm}
  \| v \|^2_L = \int_M (Lv, v) dx.
\end{equation}
Here $(\cdot, \cdot)$ is the inner product on $\R^n$ and $L$ is a positive
definite self-adjoint differential operator. This equation plays a central role
in computational anatomy, where the distance between an image and a template is
measured as a length of a geodesic connecting them. See~\cite{Younes:2009ky} for
details.

Later in this paper we will use the \emph{flat operator} instead of $L$:
\begin{equation}
  \label{eq:flat}
  \flat: v \mapsto v^{\flat} \in \Omega^1(M), \quad \langle v^{\flat}, u \rangle
  = (Lv, u), \; \text{for any } v,u \in \Vect(M),
\end{equation}
where $\Omega(M)$ is the space of one-forms on $M$ and $\langle \cdot, \cdot
\rangle$ is the pairing of a one-form and a vector field.

The EPDiff equations can be derived from the following variational principle:
\begin{equation}
  \label{eq:EPDiffVar}
  \delta \int_0^1 \int_M \langle v^{\flat}, v \rangle dx dt = 0, \quad \delta v
  = \dot\xi + [v,\xi], \quad \xi \big |_{t=0} = \xi \big |_{t=1} = 0.
\end{equation}
The constraints on $\delta v$ are called \emph{Lin constraints}
in~\cite{Marsden:1999wk} and are due to the fact that the variations are taken
along a path on the Lie group $\Diff(M)$ while $v$ belongs to its Lie
algebra. Substituting the expression for $\delta v$ into the integral and using
the fact that the commutator of vector fields is the Lie derivative
$[v, u] = L_v u$, we get
\begin{equation*}
  \int_0^1 \int_M \left( \langle v^{\flat}, \dot\xi \rangle + \langle v^{\flat}, L_v
  \xi \rangle \right) dxdt = 0,
\end{equation*}
which after integration by parts becomes
\begin{equation*}
  \int_0^1 \int_M \langle -\dot v^{\flat} - L_v v^{\flat} - v^{\flat} \div v,
  \xi \rangle dx dt = 0.
\end{equation*}
Thus, we obtain the EPDiff equation:
\begin{equation}
  \label{eq:EPDiff}
  \dot v^{\flat} + L_v v^{\flat} + v^{\flat} \div v = 0.
\end{equation}

Later on in this paper we will consider a special case of the EPDiff equation
when $\dim M = 1$ and $Lv = v - \alpha^2 \partial^2_x v$. In this case the
EPDiff equation becomes the \emph{Camassa-Holm (CH) equation}:
\begin{equation}
  \label{eq:CH}
  \dot m + (mv)_x + mv_x = 0, \quad m = v - \alpha^2 v_{xx},
\end{equation}
which is a well known model for waves in shallow water
(see~\cite{Camassa:1994uj}). This equation is completely integrable and has
soliton solutions called \emph{peakons} which have a discontinuity in the first
derivative. Due to this, solving the CH equation numerically can be challenging.

\subsection{Overview of the method}
To construct a discrete version of the EPDiff equation, we will use the method
introduced in~\cite{Pavlov:2010dd} to discretize the Euler equation of ideal
incompressible fluid. In this paper, however, we extend this method to apply to
the whole space of diffeomorphisms in a pseudospectral representation of the
velocity.

According to this method we replace the group of diffeomorphisms with a group of
matrices, on which we will construct a Lagrangian system with nonholonomic
constraints. The derivation of the finite-dimensional version of the EPDiff
equation on a matrix group will closely follow the derivation of the EPDiff
equation presented above.

\section{General method}

\subsection{Discrete diffeomorphisms}
Following~\cite{Pavlov:2010dd} we will replace a diffeomorphism $g\in\Diff(M)$
by a linear operator $U_g$:
\[
U_g:L_2(M)\to L_2(M),\quad U_g:\phi\mapsto\phi\circ g^{-1},
\]
where $L_2(M)$ denotes the space of square-integrable functions on $M$.
We will consider a finite-dimensional linear operator $q$ as an approximation to
the diffeomorphism $g$ and write $q\leadsto g$ if $q$ approximates $U_g$.

To discretize the linear operator $U_g$ we first need to discretize the space
where it acts, i.e.\ the space of $L_2$ functions on $M$. To do this we fix a
family finite-dimensional spaces $\F_N\subset L_2(M)$, $\dim\F_N = N$ and two
families of operators
\[
\mathbf D_N: L_2(M)\to\F_n,\quad\text{and }\mathbf R_N:\R^N\to\F_N.
\]
We will call the family $\mathbf D_N$ \emph{a discretization} of $L_2(M)$ if
for any function $\phi\in L_2(M)$ the sequence $\phi_N = \mathbf D_N\phi$
converges to $\phi$ as $N\to\infty$. We will call the $N$-dimensional vector
$\phi^d_N = \mathbf R^{-1}_N\phi_{N}$ \emph{a discrete function} and the operator $\mathbf R_N$
\emph{a reconstruction operator}.

Now we can define a discrete diffemorphism as a linear operator acting on
discrete functions:
\begin{Def}
Let $\mathbf D_N$ be a discretization of $L_2$ and $\mathbf R_N$ a
family of reconstruction operators. We will say that a family of linear operators
$q_N:\R^N\to\R^N$ is an \emph{approximation} to a diffeomorphism $g\in\Diff(M)$ and
write $q_N\leadsto g$ if for any function $\phi\in L_2(M)$ we have:
\begin{equation}
  \mathbf{R}_N q_N \mathbf{R}_N^{-1}\mathbf{D}_N\phi\to U_g\phi,\quad\text{when } N\to\infty.
\end{equation}
\end{Def}

\begin{figure}[t]
\[
\begin{tikzcd}[column sep=large]
C^1(M) \arrow{r}{\Diff(M)} \arrow{d}{\mathbf D_{N}} & C^1(M) \arrow{d}{\mathbf D_N} \\
\F_N \arrow{r} \arrow[transform canvas={xshift=0.7ex}]{d}{\mathbf R_N^{-1}} & \F_N \arrow[transform canvas={xshift=0.7ex}]{d}{\mathbf R_N^{-1}} \\
\R^N \arrow[transform canvas={xshift=-0.7ex}]{u}{\mathbf R_N} \arrow{r}{\D(M)} &
\R^N \arrow[transform canvas={xshift=-0.7ex}]{u}{\mathbf R_N}
\end{tikzcd}
\]
\caption{Discretization and reconstruction operators. Here $\F_N$ is a space of
  discrete functions and $\mathbf R_N$ is a bijection. $\D(M)$ is the group of
  discrete diffeomorphisms, which is a finite-dimensional group of linear
  operators.}
\label{discr}
\end{figure}

Thus, to discretize the group of diffeomorphisms we first need to choose a
discretization of $L_2$ functions and then fix a group of linear operators acting on the
discrete functions. Different methods can be used for both of these steps,
we will describe one such method in more detail below. After the set of
discrete diffeomorphisms has been chosen we will denote it $\D(M)$. The
relationship between $\D(M)$ and $\Diff(M)$ is illustrated by the diagram in
Figure~\ref{discr}. Note that the diagram doesn't commute.

\subsection{Discrete vector fields}
To define a discrete vector field let's consider a smooth path $q_t\in\D(M)$ of
discrete diffeomorphisms. A discrete function $\phi^d_0$ is transported by the
flow $q_t$:
\[
\phi^d_t = q_t\phi^d_0.
\]
It satisfies the equation
\begin{equation}
  \dot\phi^d_t = \dot q_t\phi^d_0 = \dot q_t q^{-1}\phi^d_t = U_t\phi^d_t,
\end{equation}
where $U_t = \dot q_t q_t^{-1}$. Note, that this equation is analogous to the
advection equation
\[
\dot\phi_t = -L_{u_t}\phi_t,
\]
where $L_{u_t}$ is the Lie derivative along the vector field $u_t$. Thus, the
linear operator $U_t = \dot q_t q_t^{-1}$ can be considered a discretization of
the Lie derivative, which brings us to the following definition:
\begin{Def}
Let $\mathbf D_N$ be a discretization of $L_2$ and $\mathbf R_N$ a
family of reconstruction operators. We will say that a family of linear operators
$U_N:\R^N\to\R^N$ is an approximation to a vector field $u\in\Vect(M)$ and
write $U_N\leadsto u$ if for any function $\phi\in C^1(M)$ we have:
\begin{equation}
  \mathbf{R}_N U_N \mathbf{R}_N^{-1}\mathbf{D}_N\phi\to -L_u\phi,\quad\text{when } N\to\infty,
\end{equation}
where convergence is assumed to be in $L_2$ norm.
\end{Def}

Now, if we assume that the discrete diffeomorphisms $\D$ from a Lie group, we
can see that the space of discrete vector fields, which we will denote by
$\dTO$, is the Lie algebra of $\D$. Moreover, the commutator $[U, V] = UV-VU$ of
two discrete vector fields is an approximation to the commutator of the
continuous vector fields $u$ and $v$, assuming $U\leadsto u$ and $V\leadsto v$.
If the space of discrete functions $\F$ has dimension $N$, the
space of discrete vector fields may have dimension as large as $N^2$. To make
the discretization computationally tractable we will restrict the discrete
vector fields to belong to a space $\S$ of dimension $O(N)$ instead. However,
the space $\S$ is likely not closed under commutators, $[\S,\S]\nsubseteq \S$,
and therefore we cannot restrict discrete diffeomorphisms to a subgroup of
$\D$. A method to construct a constrained set $\S$ will be outlined below.

For every vector field $v \in \Vect(M)$ we will be able to construct its
discrete version $V \in \S$, thus we will define an operator
$\mathbf S: \Vect(M) \to \S$. We will require this operator to be
right-invertible, so any matrix $V \in \S$ can be reconstructed into a vector
field. Later in this paper we will use a pseudospectral representation in which
a vector field on a circle is represented by its values at $N$ points. The
operator $\mathbf S$ will be defined in~\eqref{eq:S}.

Note that the matrices in the commutator space $[\S, \S]$, however, cannot be
identified with continuous vector fields. See figure~\ref{discrVF}.


\begin{figure}[t]
\[
\begin{tikzcd}[column sep=large]
C^\infty(M) \arrow{r}{\Vect(M)} \arrow{d}{\mathbf D_{N}} & C^\infty(M) \arrow{d}{\mathbf D_N} \\
\F_N \arrow{r}{\mathfrak D}
\arrow[transform canvas={xshift=0.7ex}]{d}{\mathbf R_N^{-1}} &
\F_N \arrow[transform canvas={xshift=0.7ex}]{d}{\mathbf R_N^{-1}} \\
\R^N \arrow[transform canvas={xshift=-0.7ex}]{u}{\mathbf R_N} \arrow{r}{\S} &
\R^N \arrow[transform canvas={xshift=-0.7ex}]{u}{\mathbf R_N}
\end{tikzcd}
\]
\caption{Discretization of vector fields. Here $\S$ is a set of linear operators
on $\R^N$ representing vector fields and $\mathfrak D$ is the Lie algebra of the
group of discrete diffeomorphisms $\D$.
}
\label{discrVF}
\end{figure}

\subsection{Discrete forms and flat operator}
Let's assume the space $\Vect(M)$ is equipped with an inner product
$(\cdot,\cdot)$. A discrete version of this inner product can be defined as
follows:
\begin{Def}
  A family of Hermitian forms $(\cdot, \cdot)^d_N$ on $\dTO_N$ is said to be an
  approximation to the inner product $(\cdot,\cdot)$ if for any pair of vector
  fields $u,\,v\in\Vect(M)$ and its discretization $U_N\leadsto u$, $V_N\leadsto v$,
  such that $U_N\in\S$, $V_N\in\S\cup [\S,\S]$ we have
  \begin{equation}
    (U_N,V_N)^d_N\to (u,v),\quad\text{when } N\to\infty.
  \end{equation}
\end{Def}
Later on we will omit the superscript $d$ in the formula above and simply write
$(U, V)$ for the discrete inner product.

An inner product $(\cdot,\cdot)$ on $\Vect(M)$ defines \emph{a flat operator}
\[
\flat: u\mapsto u^{\flat}\in\Omega^1(M),\quad (u,v) = u^{\flat}(v),\,\text{for any } v\in\Vect(M),
\]
where $\Omega^1(M)$ is the space of one-forms on $M$.

Following~\cite{Pavlov:2010dd} we define a discrete one-form as an object dual to the discrete
vector fields, i.e.\ as a matrix $F$ and a pairing
\[
\langle F, U \rangle = \Tr(FU^*).
\]
This definition of the pairing allows us to define a discrete flat operator
$\flat: U\mapsto U^{\flat}$ as
\begin{equation}
  \flat: U\mapsto U^{\flat},\quad (U,V) = \Tr(U^{\flat} V^*),\,\text{for any } v\in\Vect(M).
\end{equation}


\subsection{Lagrangian mechanics on the group of discrete diffeomorphisms}

Our goal is to construct a Lagrangian system on the group $\D(M)$ of discrete
diffeomorphisms approximating a certain continuous dynamics on $\Diff(M)$. To do this,
we will construct a Lagrangian of the form (see
section~\ref{discreteflatoperator} for an explicit construction of the flat
operator)
\begin{equation}
  L(U) = \frac12 \langle U^{\flat}, U \rangle
\end{equation}
and derive the dynamics from the Lagrange-D'Alembert principle:
\begin{equation}
  \delta\int_0^1 L(U) dt = 0, \quad \delta q q^{-1} \in \S, U \in S, \quad
  \delta q(0) = \delta q(1) = 0.
\end{equation}
The equations describing the dynamics can be easily derived as follows: first,
since $U = \dot q q^{-1}$ we can show that $\delta U$ has to satisfy the Lin
constraint:
\begin{equation}
  \delta U = \dot B + [U, B], \quad \text{where } B = \delta q q^{-1}.
\end{equation}
Second, substituting the Lin constraint into the expression for $\delta L(U)$ we
get
\begin{equation}
  \delta L(U) = \frac12 \langle U^{\flat}, \dot B + [U, B] \rangle.
\end{equation}
Thus the Lagrange-D'Alembert principle may be written as
\[
\int_0^1 \Tr \big({U^{\flat}}^{*} (\dot B + [U, B])\big) dt = 0, \quad \text{for
any } B\in S,
\quad B \big|_{t=0} = B \big|_{t=1} = 0,
\]
which after integration by parts and rearrangement by permuting under the trace
yields
\begin{equation}
\label{conttime}
  \langle \dot U^{\flat} + [U^{*}, U^{\flat}], B \rangle = 0, \quad \text{for any } B
  \in \S. 
\end{equation}

\subsection{Discrete time}
\label{sec:discrete-time}
To discretize time we consider the dynamics is given as a discrete path
$q_0, \ldots, q_K$ on $\D(M)$, where motion is sampled at regular time intervals
$t_k = k\cdot dt$, where $dt$ is a time step.  For a given pair of
configurations $q_k,\, q_{k+1}$ we use one of the following ways to define matrix $U$
for discrete time:
\begin{align*}
q_{k+1}-q_k & =  dt \; U_k \; q_k,                  & \text{(explicit Euler)}, \\
q_{k+1}-q_k & = dt \; U_k \; q_{k+1},              & \text{(implicit Euler)}, \\
q_{k+1}-q_k & = dt \; U_k\; \frac{q_k+q_{k+1}}{2}, & \text{(midpoint rule)}, \\
(q_{k+1} - q_k)\frac{q_{k+1}^{-1} + q_k^{-1}}{2} & = dt U_k, &
\text{(average explicit-implicit)}.
\end{align*}
These four approaches to discretization result in the following four
representations of the discretized variational relations:
\begin{enumerate}
  \item {\it Explicit Euler.} In this case,
$U_k=(q_{k+1}-q_k)/dt\;q_k^{-1}$. The variation $\delta_k U_k$ and
$\delta_{k+1}U_k$ with respect to $q_k$ and $q_{k+1}$ respectively become:
  \[ \delta_k U_k=- \frac{1}{dt} \delta q_k q_k^{-1}-\frac{q_{k+1}-q_k}{dt}
q_k^{-1}\delta q_k q_k^{-1},\]
  \[ \delta_{k+1} U_k=\frac{1}{dt} \delta q_{k+1}q_k^{-1}. \] If we denote,
similarly to the continuous case, $B_k=\delta q_k q_k^{-1}$, we get:
  \[ \delta_k U_k=-\frac{B_k}{dt}+U_kB_k \] and
  \[ \delta_{k+1} U_k=\frac{B_{k+1}}{dt}+B_{k+1}U_k.\]
  
  \item {\it Implicit Euler.} In this case
$U_k=\frac{q_{k+1}-q_k}{dt}\;q_{k+1}^{-1}$. It yields: \[ \delta_kU_k=-
\frac{1}{dt} \delta q_kq_{k+1}^{-1} \] and
  \[ \delta_{k+1}U_k= \frac{1}{dt} \delta
q_{k+1}q_{k+1}^{-1}-\frac{q_{k+1}-q_k}{dt}q_{k+1}^{-1}\delta
q_{k+1}q_{k+1}^{-1}. \] Similarly to the previous case we now obtain:
  \[ \delta_kU_k=-\frac{B_k}{dt}-B_kU_k,\] and
  \[ \delta_{k+1}U_k=\frac{B_{k+1}}{dt}-U_kB_{k+1}. \]
  
  \item {\it Midpoint.} The Eulerian velocity between $q_k$ and $q_{k+1}$ is now
expressed as $U_k=2\frac{q_{k+1}-q_k}{dt}(q_{k+1}+q_k)^{-1}.$ Thus,
  \begin{align*} \delta_k U_k & =-2\frac{\delta q_k}{dt}(q_{k+1}+q_k)^{-1} \\ &
\qquad -2\frac{q_{k+1}-q_k}{dt}(q_{k+1}+q_k)^{-1}\delta q_k(q_{k+1}+q_k)^{-1} \\
& =-\frac1{dt}(2B_k+dt U_kB_k)q_k(q_{k+1}+q_k)^{-1} \\ &
=-\frac1dt(\Id+\frac1{dt} 2 U_k)\;B_k\;(\Id-\frac1{dt} 2 U_k).
  \end{align*}

\item {\it Average Explicit-Implicit.} Here the velocity between $q_k$ and
  $q_{k+1}$ is expressed as an average of the velocities computed with explicit
  and implicit rules:
  \begin{equation}
    U_k = \frac12\frac{1}{dt}(q_{k+1} - q_k) (q_k^{-1} + q_{k+1}^{-1}).
  \end{equation}
In this case the variations $\delta_{k,k+1} U_k$ are also averages of the
corresponding variations:
\begin{align}
  \delta_k U_k =     & -\frac{B_k}{dt} +\frac12 [U_k, B_k], \\
  \delta_{k+1} U_k = & \frac{B_{k+1}}{dt} + \frac12 [B_{k+1}, U_k].
\end{align}

\end{enumerate}

Now that we have these four different ways to compute variations of $U_k$, we
can proceed to derive the corresponding discrete Lagrange-D'Alembert
equations. e define the discrete-space/discrete-time Lagrangian
$L_d(q_k,q_{k+1})$ as
\[
L_d(q_k,q_{k+1})=L(U_k).
\]
The discrete action $\A_d$ along a discrete path is then simply the sum of all
pairwise discrete Lagrangians:
\[ \A_d(q_0,\ldots,q_K)=\sum_{k=0}^{K-1} L_d(q_k,q_{k+1}). \]
We can now use the Lagrange-d'Alembert principle that states that $\delta
\A_d=0$ for all variations of the $q_k$ (for $k=1,\dots,K-1$, with $q_0$ and
$q_K$ being fixed) in $S_q$ while $A_k$ is restricted to $\S$.

Setting the variations of $\A_d$ with respect to $\delta q_k$ to zero for $k\in[1,K-1]$ yields:
\begin{equation}\label{DEL}
\delta_k\left\langle U^\flat_{k-1},U_{k-1}\right\rangle+\delta_k\left\langle
U^\flat_k,U_k\right\rangle=0.
\end{equation}

Now, let's solve it for $U_k$ in the explicit case. Substituting the expressions for
$\delta_kU_k$ and $\delta_k U_{k-1}$ yields:
\[
\Tr\bigl[-U_k^\flat (B_k^*+dt B_k^* U_k^*)+U_{k-1}^\flat (B_k^*+dt U_{k-1}^*B_k^*)\bigr]=0.
\]
Denoting $\dot U_k^\flat=(U_k^\flat-U_{k-1}^\flat)dt^{-1}$ we can rewrite the last
equation as
\begin{equation}
\label{eq:update1}
\Tr[(\dot U_k^\flat -U_k^\flat  U_k^*+U_{k-1}^* U_{k-1}^\flat )B_k^* ]=0.
\end{equation}

Let's fix a basis $B_k$ of the space $\S$, i.e. any matrix $U\in \S$ can be written as
\begin{equation}
\label{eq:2}
U = \sum_k X_k B_k.
\end{equation}
Now let's rewrite the equation~(\ref{eq:update1}) in the coordinates $X$.
First, we have
\[
  U^*_{k-1} U^{\flat}_{k - 1} - U_k^{\flat} U_k^{*} =
  \sum \bar X^{k-1}_i B_i^* X^{k-1}_j B_j^{\flat} -
    X^k_i B_i^{\flat} \bar X^k_j B_j^*
\]
Now, if we denote by $A\cdot B$ the Frobenius product of $A$ and $B$, we can write
\[
\left( U^*_{k-1} U^{\flat}_{k - 1} - U_k^{\flat} U_k^{*} \right)\cdot \bar B_p =
\sum \bar X^{k-1}_i X^{k-1}_j (B_i^* B_j^{\flat})\cdot \bar B_p -
X^k_i \bar X^k_j (B_i^{\flat} B_j^*)\cdot \bar B_p.
\]
Let's denote
\[
(B_i^* B_j^{\flat})\cdot \bar B_p = \Tr (B_i^{*} B_j^{\flat} B_p^{*}) =
(\bar B_{i} \bar B_{p})\cdot B_{j}^{\flat} = \langle B_j^{\flat}, B_i B_p \rangle = C_{ijp},
\]
\[
(B_i^{\flat} B_j^*)\cdot \bar B_p = \Tr(B_i^{\flat} B_j^{*} B_p^{*}) =
(\bar B_{p} \bar B_{j})\cdot B_{i}^{\flat} = \langle B_i^{\flat}, B_p B_j \rangle = D_{ijp}
\]
and
\[
B_i^{\flat}\cdot \bar B_p = E_{ip}.
\]
Then the update rule for the explicit case can be written as
\begin{multline}
\label{eq:explupdate}
\sum_i E_{ip} X^k_{i} - \sum_i E_{ip} X^{k-1}_i +\\
\sum_{i, j} C_{ijp} \bar X^{k-1}_i X^{k-1}_j -
\sum_{i, j} D_{ijp} X^{k}_i \bar X^{k}_j = 0
\end{multline}
Similarly, in the implicit case we get
\begin{multline}
\label{eq:implupdate}
\sum_i E_{ip} X^k_{i} - \sum_i E_{ip} X^{k-1}_i +\\
\sum_{i, j} C_{ijp} \bar X^{k}_i X^{k}_j -
\sum_{i, j} D_{ijp} X^{k-1}_i \bar X^{k-1}_j = 0.
\end{multline}
In the average explicit-implicit case the update rule is the average of the two
formulas above. The midpoint case yields third order terms in $U$ and it's not
considered here.

\section{Pseudospectral discretization}

\subsection{Discrete functions and vector fields}

To illustrate the method, we consider the following case of pseudospectral
discretization. Let's define the space $\S$ of discrete vector fields on $S^1$ using a
pseudospectral representation. Note, that a matrix $U \in \S$ is an
approximation to an operator of Lie derivative $L_u$:
\[
L_u \phi_c = \phi_c' u.
\]
Now we will consider a continuous test function $\phi^c$ being represented by its
truncated Fourier series, i.e. by a vector $(\phi_{-N}, \ldots, \phi_N)$, where
\[
\phi_k = \frac{1}{2\pi}\int_{-\pi}^{\pi} e^{- \imath k x} \phi^c(x) dx, \quad k = -N,
\ldots, N.
\]
We will denote by $D$ the operator of differentiation in the truncated Fourier
space, i.e.
\begin{equation}
  (D\phi)_k = D_{kk}\phi_k = \imath k \phi_k.
\end{equation}

If we know values $u_k$ of a vector field $u(x)$ at points $x_k = -\pi + k\cdot
2\pi/(2N + 1)$
we can define a discrete version of the multiplication operator $\phi^c\mapsto
\phi^c\cdot u$ as
\begin{equation}
  M = F T_u F^{-1},
\end{equation}
where $F$ is the discrete Fourier transform and $(T_u)_{ij} = \delta_{ij} u_i$.

Now, the space $\S$ of discrete vector fields is spanned by matrices $B_k$'s of the form
\begin{equation}
  B_k = M_k D,
\end{equation}
where
\begin{equation}
  M_k = F I_k F^{-1}, \quad (I_k)_{ij} = \delta_{ij}\delta_{ik}.
\end{equation}

To summarize, our discretization consists of the following:
\begin{enumerate}
\item Space of functions
\begin{equation}
 \F_N: \quad \left\{\phi(x) \mid \phi(x) = \sum_{k = -N}^N \phi_k e^{\imath k
     x}\right\}.
\end{equation}

\item Discretization operator:
\begin{equation}
\mathbf D_N: \quad \phi(x) \mapsto (\phi_{-N}, \ldots, \phi_N), \quad \phi_k =
\frac{1}{2\pi}\int_{-\pi}^{\pi} \phi(x) e^{-\imath k x}
\end{equation}

\item Reconstruction operator:
  \begin{equation}
    \mathbf R_N: (\phi_{-N}, \ldots, \phi_N) \mapsto \phi(x) = \sum_{k=-N}^N \phi_k
    e^{\imath kx}
  \end{equation}

\item Discretization of a vector field:
  \begin{equation}
\label{eq:S}
    \mathbf S: v(x) \mapsto F T_v F^{-1} D = \sum X_k B_k,
  \end{equation}
where
\begin{equation}
  T_v =
\begin{pmatrix}
u(x_{-N}) & \ldots & 0 & \ldots & 0 \\
0 & \ldots & u(x_0) & \ldots & 0 \\
0 & \ldots & 0 & \ldots & u(x_N)
\end{pmatrix},
\end{equation}
\begin{equation}
  X_k = u(x_k).
\end{equation}

\end{enumerate}

\subsection{Discrete flat operator}
\label{discreteflatoperator}
Let's now define a flat operator, which is the key ingredient of the method. To
define a pairing between discrete vector fields $U$ and $V$ let's note that
since $U \phi \approx L_u \phi^c$ we have for $e_k = (0, \ldots, 1, \ldots, 0)$:
\begin{equation}
  U e_k \approx -L_u e^{\imath k x} = D_{kk} u e^{\imath k x},
\end{equation}
where $\approx$ is defined in the sense of $L_2$ norm.
If a function $\phi^c$ is represented by a vector $\phi$ then $\phi_0 \approx
\int \phi$. Thus, $(U e_k)_0$ is an approximation to the $-k$-th Fourier
coefficient of $u$ multiplied by $D_{kk}$:
\begin{equation}
  (U e_k)_0 = U_{0k} \approx D_{kk} \int u e^{\imath k x}.
\end{equation}
Therefore, we can define a flat operator through the following pairing:
\begin{equation}
  \langle U^\flat, V \rangle = \sum_k \frac{U_{0k}}{D_{kk}}\frac{\bar V_{0k}}{\bar D_{kk}} +
  \alpha \sum_k U_{0k}\bar V_{0k} = \sum_k U_{0k}\bar V_{0k} \left(\alpha - D_{kk}^{-2}\right).
\end{equation}
It's worth noting that the pseudospectral discretization allows us to construct
a flat operator in a much more straightforward way than, for example,
discretization described in~\cite{Pavlov:2010dd}.

\subsection{Update rule}
Now, let's compute the update rule for the \emph{explicit} case.
\begin{theorem}
  The update rule in the explicit and implicit cases are given by the
  formulas~\eqref{eq:explupdate} and~\eqref{eq:implupdate}, where
  \begin{equation}
    \label{eq:C}
    \sum_{i, j} \bar X_i X_j C_{ijp} = \frac{1}{N} (F^{-1} \bar D F \bar X)_p
    (\overline{F^{-1} H F \bar X})_p \approx \frac{1}{N} u_x m,
  \end{equation}
  \begin{equation}
    \label{eq:D}
    \sum_{i, j} \bar X_i X_j D_{ijp} = \frac{1}{N} \overline{(F^{-1} D F) (X
      \star F^{-1} H F\bar X)} \approx \frac{1}{N} \partial_x (u m)
  \end{equation}
and
\begin{equation}
  \label{eq:E}
  \sum_i X_i E_{ip} = \frac{1}{N} (\bar F^{-1} H \bar F X)_p \approx \frac{1}{N} m.
\end{equation}
\end{theorem}

\begin{proof}
We have
\begin{equation}
  \sum_{i, j} \bar X_i X_j \langle B_j^{\flat}, B_i B_p \rangle =
  \sum_{i, j} \bar X_i X_j \langle B_j^{\flat}, F T_i F^{-1} D F T_p F^{-1} D \rangle,
\end{equation}
where $X_k = u(x_k)$ (we will write $X \approx u$ in this case).
Also,
\begin{equation}
  \langle B_j^{\flat}, U \rangle = \sum_s (F I_j F^{-1} D)_{0s}\bar U_{0s}(\alpha - D_{ss}^{-2}) =
  \sum_s F_{0j} F^{-1}_{js} D_{ss}(\alpha - D_{ss}^{-2})\bar U_{0s}.
\end{equation}
Thus, we can write
\begin{multline*}
  \sum_{i, j} \bar X_i X_j C_{ijp} =
  \sum_{i, j} \bar X_i X_j \langle B_j^{\flat}, B_i B_p \rangle = \\
  \sum_{i, j, s} \bar X_i X_j (\bar F \bar T_i \bar F^{-1} \bar D \bar F \bar T_p \bar F^{-1} \bar D)_{0s}
  F_{0j} F_{js}^{-1}(\alpha D_{ss} - D_{ss}^{-1}) =\\
  \sum_{i, j, s, k} \bar X_i X_j \bar F_{0i} \bar F^{-1}_{ik} \bar D_{kk} \bar F_{kp} \bar F^{-1}_{ps} \bar D_{ss} F_{0j} F^{-1}_{js}
  (\alpha D_{ss} - D_{ss}^{-1}).
\end{multline*}
Since $F$ is unitary and $F_{0i} = \frac{1}{\sqrt{N}}$, we have
\begin{equation*}
  \sum_i \bar X_i \bar F_{0i} \bar F^{-1}_{ik} = \frac{1}{\sqrt{N}} (F \bar X)_{k}
\end{equation*}
and
\begin{equation*}
  \sum_j X_j F_{0j} F^{-1}_{js} = \frac{1}{\sqrt{N}} (\bar FX)_{s}.
\end{equation*}
Now we have
\begin{equation}
  \sum_{i, j} \bar X_i X_j \langle B_j^{\flat}, B_i B_p \rangle = \frac{1}{N}
  (F^{-1} \bar D F \bar X)_p (\overline{F^{-1} H F \bar X})_p,
\end{equation}
where
\begin{equation}
  H = \Id - \alpha D^2. 
\end{equation}
Since $X \approx u$ we have
\[
F^{-1} \bar D F \bar X \approx u_x
\]
and
\[
\overline{F^{-1} H F \bar X} \approx u - \alpha u_{xx}.
\]
Therefore, now we have
\begin{equation}
  \sum_{i, j} \bar X_i X_j C_{ijp} \approx \frac{1}{N} u_x m,
\end{equation}
where $m = u - \alpha u_{xx}$.

Similarly,
\begin{multline*}
  \sum_{i, j} \bar X_i X_j D_{ijp} =
  \sum_{i, j} X_i \bar X_j \langle B_i^{\flat}, B_p B_j \rangle = \\
  \sum_{i, j, s} X_i \bar X_j (\bar F \bar I_p \bar F^{-1} \bar D \bar F \bar I_j \bar F^{-1} \bar D)_{0s}
  F_{0i} F^{-1}_{is} (\alpha D_{ss} - D_{ss}^{-1}) = \\
  \sum_{i, j, s} X_i \bar X_j (\bar F_{0p} \bar F^{-1}_{pk} \bar D_{kk} \bar F_{kj} \bar F^{-1}_{js} \bar D_{ss})
  F_{0i} F^{-1}_{is} (\alpha D_{ss} - D_{ss}^{-1}).
\end{multline*}
We have
\begin{equation*}
  \sum_{i, s} X_i \bar F^{-1}_{js} F_{0i} F_{is}^{-1} H_{ss} = \frac{1}{\sqrt{N}} (\bar F^{-1} H \bar F X)_j,
\end{equation*}
thus
\begin{multline*}
  \sum_{i, j} X_i \bar X_j \langle B_i^{\flat}, B_p B_j \rangle =
  \frac{1}{N} \sum_{j, s} \bar X_j (\bar F^{-1} H \bar F X)_j
  \bar F_{kj} \bar D_{kk} \bar F_{pk}^{-1} = \\
  \frac{1}{N} (\bar F^{-1} \bar D \bar F) (\bar X \star \overline{F^{-1} H F\bar X}) =
  \frac{1}{N} \overline{(F^{-1} D F) (X \star F^{-1} H F\bar X)},
\end{multline*}
where $(X \star Y)_i = X_i Y_i$.
Again, since $X \approx u$ we have
\[
X \star F^{-1} H F\bar X \approx u m
\]
Therefore,
\begin{equation}
  \sum_{i, j} \bar X_i X_j D_{ijp} \approx \frac{1}{N} \partial_x (u m) = \frac{1}{N} (u_x m + u m_x).
\end{equation}

Finally, we compute $\sum_i X_i E_{ip} = \sum_i X_i \langle B_i^{\flat}, B_p \rangle$:
\begin{multline*}
  \sum_i X_i \langle B_i^{\flat}, B_p \rangle =
  \sum_{i, s} X_i (F I_i F^{-1} D)_{0s} (\bar F I_p \bar F^{-1} \bar D)_{0s} (\alpha - D_{ss}^{-2}) = \\
  \sum_{i, s} X_i F_{0i} F^{-1}_{is} D_{ss} \bar F_{0p} \bar F^{-1}_{ps} \bar D_{ss} (\alpha - D_{ss}^{-2}) =
  \frac{1}{N} (\bar F^{-1} H \bar F X)_p \approx \frac{1}{N} m.
\end{multline*}
\end{proof}

\section{Results}
\label{sec:results}
\begin{figure}[t]
\includegraphics[width=200pt]{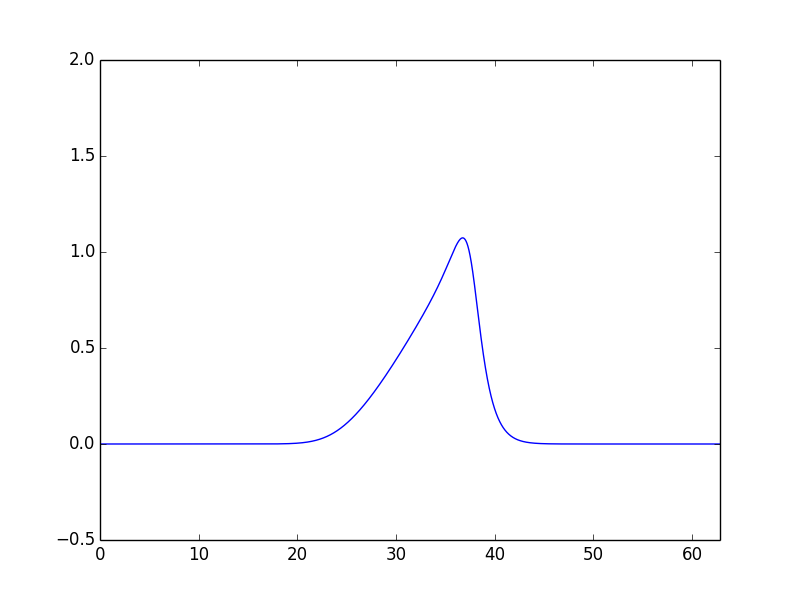}
\includegraphics[width=200pt]{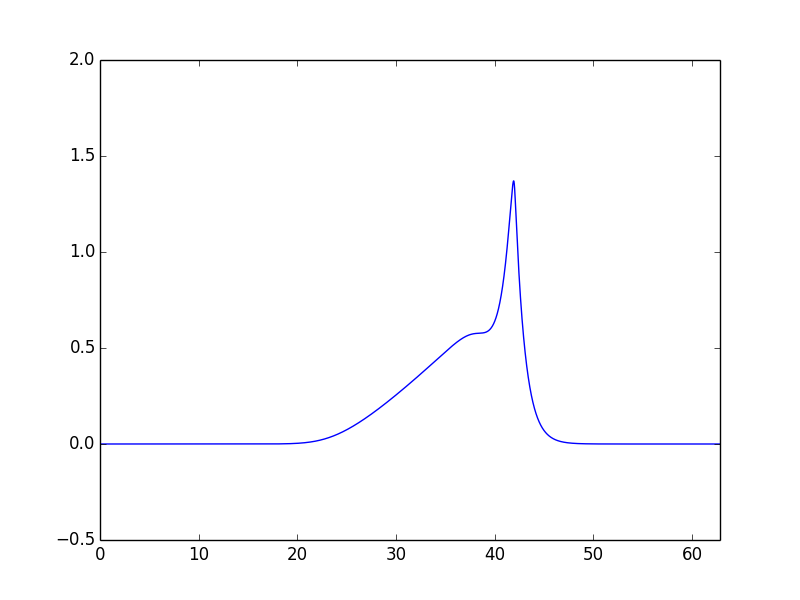}
\includegraphics[width=200pt]{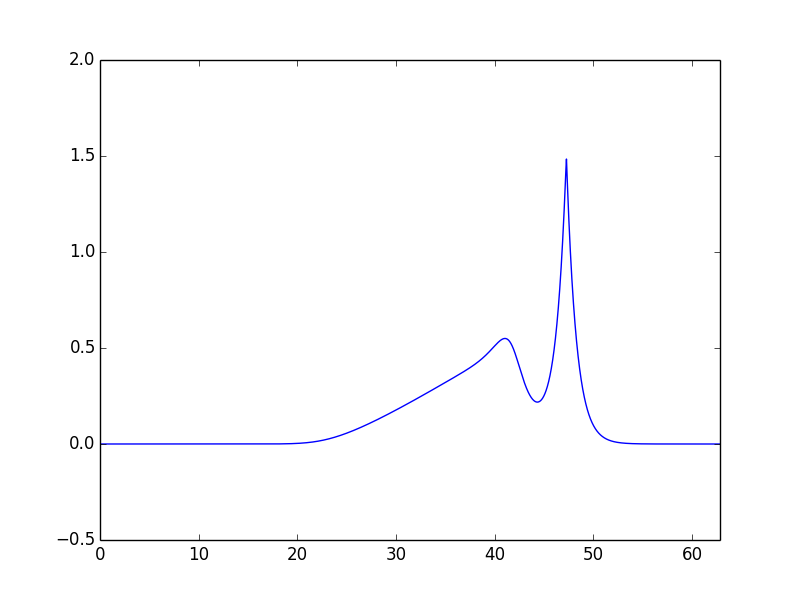}
\includegraphics[width=200pt]{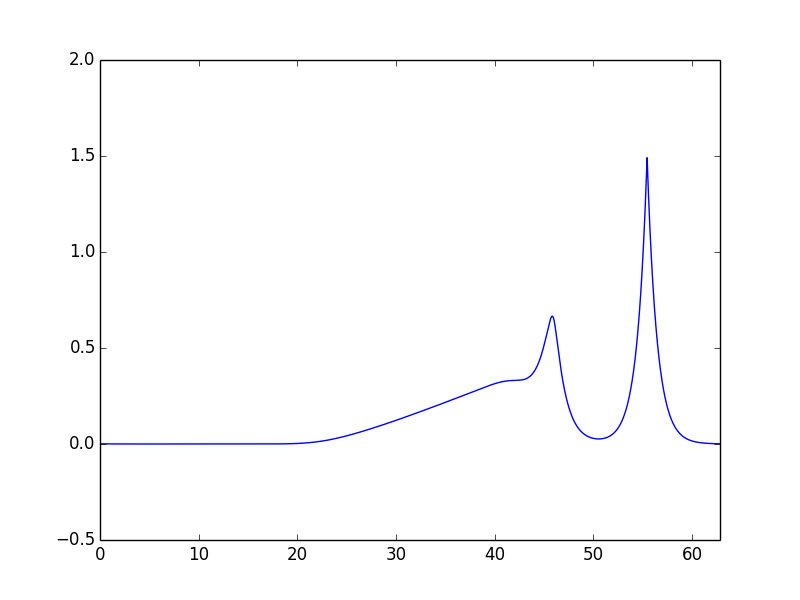}
\caption{Formation of peakons from a gaussian initial condition}
\label{peakonformation}
\end{figure}

\begin{figure}[h]
\centering
\includegraphics[width=200pt]{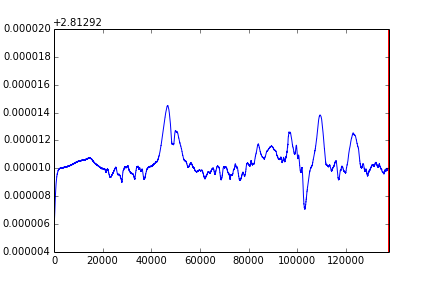}
\caption{Energy behavior for peakon formation}
\label{peakonenergy}
\end{figure}

We have implemented our method for the explicit, implicit and the average cases. In
all numerical tests we see the energy decreasing in the explicit case and
increasing in the implicit case. In the average explicit-implicit case however the
energy is stable. That is, the energy is oscillating around its correct value (see
figure~\ref{peakonenergy}). This behavior is different from energy behavior of a
variational integrator. This difference is a result of imposing nonholonomic
constraints. The same behavior has also been observed in other systems of the
same form, i.e. for the equation~\eqref{conttime} with a different flat operator.

We studied different cases of peakon dynamics, such as formation of peakons from
a gaussian initial condition, interaction of peakons of the same sign and peakon
collisions. Formation of peakons from a gaussian initial condition is shown in
Figure~\ref{peakonformation}. For this case we chose $\alpha=1$, $N=1000$ and
$dt=0.01$. Peakon collision remains a challenge. The simulation leads to
creation of multiple peaks, but remains stable (see
figure~\ref{peakoncollision}). The energy drops when the two peakons collide
initially, but then recovers and remains stable (see
figure~\ref{collisionenergy}).

\begin{figure}[h]
\centering
\includegraphics[width=250pt]{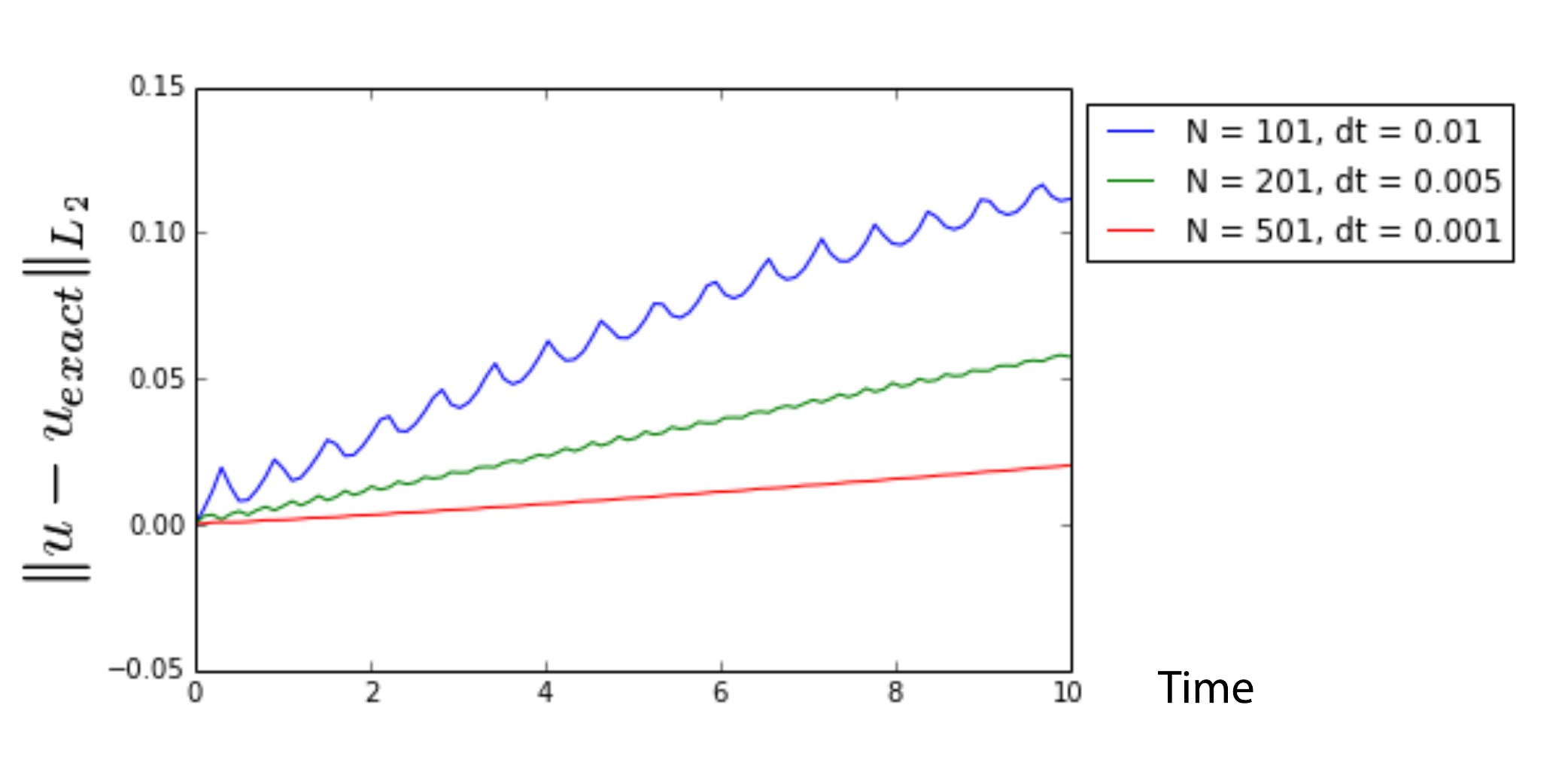}
\caption{Accuracy and convergence for a single peakon}
\label{accuracy}
\end{figure}

\begin{figure}[h]
\centering
\includegraphics[width=250pt]{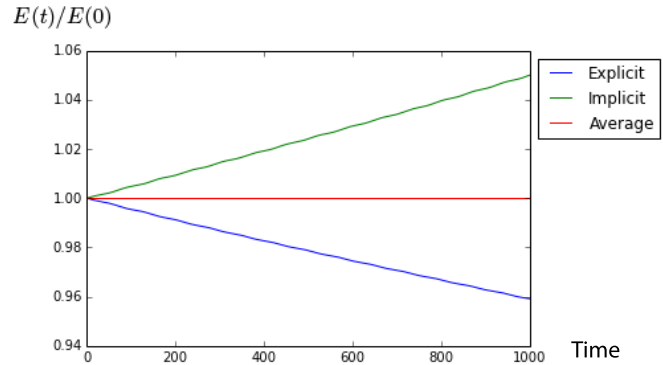}
\caption{Energy behavior for three different methods}
\label{fig:energy}
\end{figure}

\section{Conclusions and summary}
To summarize, we have developed a method of discretization for systems on the
group of diffeomorphisms. This method is presented here for the case of the
Camassa-Holm equation, but can easily be applied to other systems. The method
itself is flexible and can use different representations of vector fields
(operator $\mathbb S$ in Fig.~\ref{discrVF}). The final update rule is derived
from a variational principle with nonholonomic constraints and the resulting
energy behavior is different from that of a variational integrator. Namely, the
energy behavior depends on how the discrete velocity is computed from a pair of
configurations (see Section~\ref{sec:discrete-time}). In the average
explicit-implicit case the energy remains stable over long time.

\section{Future work}
While the time-continuous system~\eqref{conttime} is energy-preserving, the
energy behavior of the time discrete system depends on the choice of
discretization of $U$. One may use an adaptive time step method described
in~\cite{Cortes:2003wk} to construct an energy-preserving integrator. However,
the effect nonholonomic constraints have on a variational integrator remains an
open question.

\section{Acknowledgments}
This work was supported by ERC Advanced Grant FCCA \#267382 supervised by Darryl
Holm. The author is grateful to Darryl Holm, Colin Cotter, Alexis Arnaudon, Alex
Castro, Jaap Eldering, Henry Jacobs and Tomasz Tyranowski for their encouragement
and thoughtful comments.

\begin{figure}[h]
\includegraphics[width=200pt]{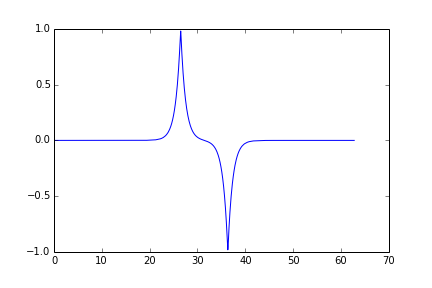}
\includegraphics[width=200pt]{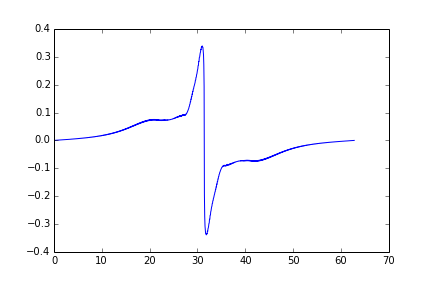}
\includegraphics[width=200pt]{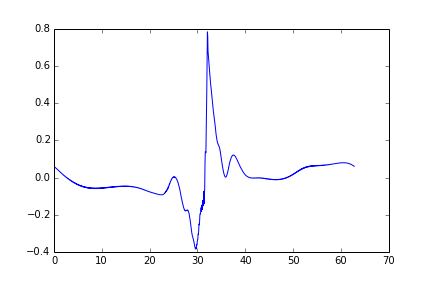}
\includegraphics[width=200pt]{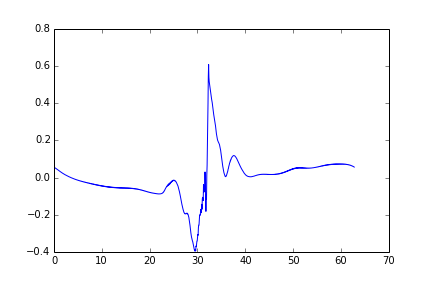}
\caption{Peakon collision sequence shows unstable behavior}
\label{peakoncollision}
\end{figure}
\begin{figure}[h]
\centering
\includegraphics[width=200pt]{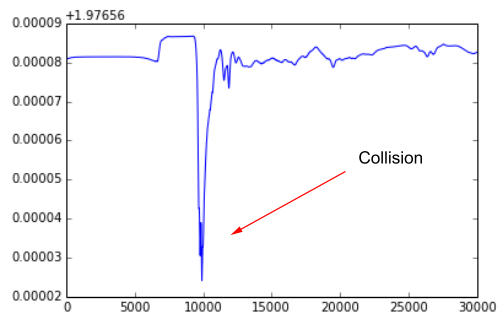}
\caption{Energy behavior for peakon collision. Energy jumps during the initial
  collision but returns to the neighborhood of its correct value after that.}
\label{collisionenergy}
\end{figure}
\clearpage

\bibliographystyle{acm}
\nocite{*}
\bibliography{References.bib}

\end{document}